\documentclass[3p,times]{elsarticle}

\usepackage{ecrc}
\usepackage{amssymb,amsmath,amsthm,mathrsfs,xcolor}
\newtheorem{theorem}{Theorem}[section]

\newtheorem{remark}{Remark}[section]

\volume{00}

\firstpage{1}

\journalname{Applied Mathematics Letters}

\runauth{J. Zuo et al.}

\jid{aml}

\jnltitlelogo{ \small Applied Mathematics Letters}

\CopyrightLine{2023}{Published by Elsevier Ltd.}

\usepackage[figuresright]{rotating}

\usepackage{color}

\begin{document}

\begin{frontmatter}




\dochead{}

\title{Long-time behavior for the Kirchhoff diffusion problem with magnetic
fractional Laplace operator}


\author[J. Zuo]{Jiabin Zuo}
\ead{zuojiabin88@163.com}
\address[J. Zuo]{School of Mathematics and Information
	Science, Guangzhou University, Guangzhou, 510006, China}

\author[J. Honda Lopes]{Juliana Honda Lopes}
\ead{juhonlopes@gmail.com}
\address[J. Honda Lopes]{Instituto de Matem\'atica e Estat\'istica, Universidade de
S\~ao Paulo, Rua do Mat\~ao, 1010, S\~o Paulo - SP CEP 05508-090,
Brazil}

\author[v,d,r,i]{Vicentiu D. R\u{a}dulescu\corref{mycorrespondingauthor}}\cortext[mycorrespondingauthor]{Corresponding author. The research of V. D. R\u{a}dulescu is supported by the grant ``Nonlinear Differential Systems in Applied Sciences" of the Romanian Ministry of Research, Innovation and Digitization, within PNRR-III-C9-2022-I8 (Grant No. 22). Jiabin Zuo is supported by the Guangdong Basic and Applied Basic Research
	Foundation (2022A1515110907) and the Project funded by China
	Postdoctoral Science Foundation (2023M730767).}
\ead{radulescu@inf.ucv.ro}
\address[v]{Faculty of Applied Mathematics, AGH University of Science and Technology, 30-059 Krak\'{o}w, Poland}
\address[d]{Brno University of Technology, Faculty of Electrical Engineering and Communication, Technick\'a 3058/10, Brno
	61600, Czech Republic}
\address[r]{Department of Mathematics, University of Craiova, Street A.I. Cuza 13, 200585 Craiova, Romania}
\address[i]{Simion Stoilow Institute of Mathematics of the Romanian Academy, Calea Grivi\c tei 21, 010702 Bucharest, Romania}
\begin{abstract}
We consider a Kirchhoff-type diffusion problem driven by the magnetic fractional Laplace operator.
The main result in this paper establishes  that infinite time blow-up cannot occur for the problem.
The proof is based on the potential well method, in relationship with energy and Nehari functionals.
\end{abstract}


\end{frontmatter}


\section{Introduction}

Let $\Omega\subset\mathbb{R}^n$ ($n>2s$) be a bounded
domain with smooth boundary.
In this paper we study the following Kirchhoff-type diffusion problem
\begin{equation}\label{problem1}
\begin{cases}
u_t+M\left(\|u\|_{X_{0,A}}^2\right)(-\Delta)_A^su= f(|u|)u,&\mbox{ in }\Omega\times(0,T),\\
u(x,t)=0,&\mbox{ in }(\mathbb{R}^n\setminus\Omega)\times(0,T),\\
u(x,0)=u_0(x),&\mbox{ in }\Omega.
\end{cases}
\end{equation}

Given $s \in (0,1)$ and $A\in
L^{\infty}_{loc}(\mathbb{R}^n)$, we define the magnetic fractional Laplace operator defined $(-\Delta)_A^s$
 by
\begin{equation}\label{operator}
(-\Delta)_A^su(x,t)=2\lim_{\varepsilon\rightarrow0^+}
\int_{\mathbb{R}^n\setminus B(x,\varepsilon)}\frac{u(x,t)-e^{i(x-y)
        \cdot A(\frac{x+y}{2})}u(y,t)}{|x-y|^{n+2s}}\ dy,
    \end{equation} for
$x\in\mathbb{R}^n$ and $u\in C_0^{\infty}(\mathbb{R}^n,\mathbb{C})$. This differential operator is weighted by a Kirchhoff-type function $M:[0,\infty)\rightarrow[0,\infty)$ (see \cite{Kirchhoff1883}), satisfying  $(M1)$-$(M2)$ below. When $A \equiv 0$ in \eqref{operator}, then we have the usual fractional Laplacian differential operator denoted by $(-\Delta)^s$. Such differential operator was studied in the context of problems in quantum mechanics and of the motion of chains or arrays of particles connected by elastic springs, as well as in the context of problems of unusual diffusion processes in turbulent fluids and of mass transport in fractured media. We refer to  \cite{Applebaum} (L\'evy processes),  \cite{Caffarelli} (nonlocal diffusions, drifts and games),  \cite{Avron,ji1,ji2, Li,wen,Radulescu} for
other classes of nonlocal operators. In all the aforementioned works, the authors deal with Schr\"odinger
operators with magnetic fields. For instance,  \cite{Li} establish the existence of nontrivial solutions to a parametric fractional Schr\"odinger equation in the case of critical or supercritical nonlinearity.
 Next, the operator $(-\Delta)_A^s$ (see \eqref{operator})
 was introduced in \cite{Avenia},  as a
 fractional counterpart of the magnetic Laplacian $(\nabla-iA)^2$,
 where $A:\mathbb{R}^n\rightarrow\mathbb{R}^n$ is a
 $L_{loc}^{\infty}$-vector potential. Zuo \& Lopes \cite{Zuo} established the existence of weak solutions to problem \eqref{problem1}. The strategy is based on the potential well method, hence they obtain global in time
 solutions and blow-up in finite time solutions. Here, we show that the global in time solutions to \eqref{problem1} can not blow-up in infinite time. We also mention  \cite{Zhou} (non-local parabolic
equation in a bounded convex domain),
\cite{Avenia}  (for the equation $(-\Delta)_A^su+u=|u|^{p-2}u$ posed in
$\mathbb{R}^3$), \cite{LWL} (fractional Choquard equation), and \cite{Toscano} (system of Kirchhoff type equations).

\section{Mathematical background and hypotheses}\label{Sec2}

The right framework for the analysis of equation \eqref{problem1} is the function space $X_{0,A}$ (hence $H_A^s(\Omega)$) defined as follows. For an open and bounded
set $\Omega\subset\mathbb{R}^n$ ($n>2s$), let $|\Omega|$ be the
measure of the set $\Omega$. By $L^p(\Omega,\mathbb{C})$ we mean the Lebesgue
space of complex valued functions with norm
$\|\cdot\|_{L^p(\Omega)}$ and inner product
$\langle\cdot,\cdot\rangle$. For $p=2$, $s\in(0,1)$ and $A\in
L^{\infty}_{loc}(\mathbb{R}^n)$, we consider the magnetic Gagliardo
semi-norm defined by
$$[u]_{H^s_A(\Omega)}^2:=\int\int_{\Omega\times\Omega}\frac{|u(x,t)-e^{i(x-y)\cdot A\left(\frac{x+y}{2}\right)}u(y,t)|^2}{|x-y|^{n+2s}}\ dx\ dy.$$

Hence we consider the space $H_A^s(\Omega)$ of functions $u\in
L^2(\Omega,\mathbb{C})$ with $[u]_{H^s_A(\Omega)}<\infty$ and furnished with the norm
$\|u\|_{H^s_A(\Omega)}:=(\|u\|^2_{L^2(\Omega)}+[u]^2_{H^s_A(\Omega)})^{\frac{1}{2}}.$ Referring to  \cite{Fiscella,Vecchi}, we  define
$X_{0,A}:=\{u\in H_A^s(\mathbb{R}^n): u=0 \mbox{ a.e. in }\mathbb{R}^n\setminus\Omega\},$
with the real
scalar product (see \cite{Avenia}) given as {\small
\begin{equation*}\label{product}
\begin{aligned}
\langle
u,v\rangle_{X_{0,A}}:=\mathcal{R}\int\int_{\mathbb{R}^{2n}}\frac{\left(u(x,t)-e^{i(x-y)\cdot
A\left(\frac{x+y}{2}\right)}u(y,t)\right)\overline{\left(v(x,t)-e^{i(x-y)\cdot
A\left(\frac{x+y}{2}\right)}v(y,t)\right)}}{|x-y|^{n+2s}}\ dx \ dy,
\end{aligned}
\end{equation*}}
where, for every $z\in\mathbb{C}$, by $\mathcal{R}z$ we mean the real
part of $z$ and by $\overline{z}$ its complex conjugate. This scalar
product induces the following norm
\begin{equation*}
\|u\|_{X_{0,A}}:=\left(\int\int_{\mathbb{R}^{2n}}\frac{|u(x,t)-e^{i(x-y)\cdot
A\left(\frac{x+y}{2}\right)}u(y,t)|^2}{|x-y|^{n+2s}}\ dx \
dy\right)^{\frac{1}{2}}.
\end{equation*}
We will use  $\xrightarrow{w}$ and $\to$ to denote  weak and strong convergences, respectively.
Our hypotheses on problem \eqref{problem1} are the following:
\begin{itemize}
    \item[(F)] $f\in
    C^1([0,\infty))$, and we can find $C>0$ and $\gamma\geq p$,
    for $p\in(2,2_s^{\ast})$ with $2_s^{\ast}=\frac{2n}{n-2s}$, such
    that
    \begin{equation}\label{f}
    \begin{array}{ll}
    |F(u)|\leq Cu^p, f(u)u^2\leq pCu^p, \mbox{ for } u\geq0,\\
    0<\gamma F(u)\leq f(u)u^2,  u^2(uf'(u)-(p-2)f(u))\geq 0 \mbox{ for }u>0,
    \end{array}
    \end{equation}
    where $F(u):=\int_0^{u}f(\tau)\tau\ d\tau$.
\end{itemize}

Further, the Kirchhoff
function $M:[0,\infty)\rightarrow[0,\infty)$ is as follows:
\begin{enumerate}
\item[(M1)] it is a continuous function and
there exist constants $m_0>0$ and
$\theta\in(1,\frac{2_s^{\ast}}{2})$ such that
    $M(u)\geq
m_0u^{\theta-1}$ for all $u\in[0,\infty)$;

\item[(M2)] there exists a constant
$\mu\in(1,\frac{2_s^{\ast}}{2})$ such that
    $\mu\mathscr{M}(u)\geq M(u)u$ for all $u\in[0,\infty)$,
 where
$\mathscr{M}(u):=\int_0^uM(\tau)\ d\tau.$
\end{enumerate}

Here, $\theta$ and $\mu$ satisfy
the condition:
(P) $2\max\{\theta,\mu\}<p<2_s^{\ast}$.

We consider the $C^1$-functional related to problem \eqref{problem1} and
defined by
\begin{equation}\label{J}
J(u):=\frac{1}{2}\mathscr{M}(\|u\|_{X_{0,A}}^2)-\int_{\Omega}F(|u|)\
dx.
\end{equation}
We have
$
\langle J'(u),\phi\rangle= M(\|u\|_{X_{0,A}}^2)\langle
u,\phi\rangle_{X_{0,A}}-\mathcal{R}\int_{\Omega}f(|u|)u\overline{\phi}\
dx,
$
for any $\phi\in X_{0,A}$, and we  introduce the Nehari functional for  \eqref{problem1} given by
$
I(u):=\langle J'(u),u\rangle.
$
Hence, we can consider the non-negative value $d$ (i.e., mountain pass level) given as $d= \inf_{u \in X_{0,A}\setminus \{0\} \,:\,  I(u)=0} J(u).$
According to Sattinger \cite{Sattinger} and  Payne \& Sattinger \cite{Payne} we know that if the initial energy $J(u_0)$ is less than the mountain pass level $d$, then the solution to problem \eqref{problem1} exists globally if it begins in the stable set
$\mathcal{W}=\{u\in X_{0,A} :I(u)>0, \, J(u)<d \} \cup \{0\},$
 and fails to exist globally if it starts from the unstable set $\mathcal{U}=\{u\in X_{0,A} :I(u)<0, \, J(u)<d \}.$
The functionals $I(\cdot)$ and $J(\cdot)$ are energy functionals of the stationary problem
\begin{eqnarray}\label{problem2}
\left\{
\begin{array}{ll}
M(\|u\|_{X_{0,A}}^2)(-\Delta)_A^su= f(|u|)u,&\mbox{ \text{in} }\Omega,\\
u(x)=0,&\mbox{ in }\mathbb{R}^n\setminus\Omega,
\end{array}
\right.
\end{eqnarray}
and we recall that $u\in X_{0,A}(\Omega)$ is a solution  to
\eqref{problem2} (i.e., stationary solution to problem
\eqref{problem1}) if $\langle J'(u),\phi\rangle=0$, namely
$
M(\|u\|_{X_{0,A}}^2)\langle
u,\phi\rangle_{X_{0,A}}-\mathcal{R}\int_{\Omega}f(|u|)u\overline{\phi}\
dx=0,
$
for all $ \phi\in X_{0,A}(\Omega)$, see also  \cite[Definition 1]{Zuo}. So, $u\in X_{0,A}(\Omega)$  solves
\eqref{problem1} on $(0,T)$ for $T>0$ if $u\in L^{\infty}(0,\infty;L^2(\Omega))\cap L^{\infty}(0,\infty;X_{0,A})
\cap C(0,\infty;L^2(\Omega))$ with $ u_t\in L^2(0,\infty;L^2(\Omega))$ and
\begin{align*} &
\mathcal{R}\int_0^T\int_{\Omega}u_t\overline{\phi}\
dx\, dt+\int_0^T    M(\|u\|_{X_{0,A}}^2)\langle
    u,\phi\rangle_{X_{0,A}}\, dt\\ &-\mathcal{R}\int_0^T\int_{\Omega}f(|u|)u\overline{\phi}\
    dx\, dt=0 \mbox{ for all $ \phi\in X_{0,A}(\Omega).$}
\end{align*}

\begin{remark}  We note that  \cite{Zuo} established the existence of a weak solution $u\in X_{0,A}(\Omega)$ to problem \eqref{problem1}, based on hypotheses (F) and (M1) (see \cite[Theorem 2.1]{Zuo} for the precise requirements). Imposing also  (M2), they  obtained that weak solutions   blow-up in a finite time, provided that the initial energy is negative (see \cite[Theorem 2.2]{Zuo}).  Further, problem \eqref{problem1}  has a weak solution for all $T>0$ (namely, a global solution) such that $u\in L^{\infty}(0,\infty;L^2(\Omega))\cap L^{\infty}(0,\infty;X_{0,A})
    \cap C(0,\infty;L^2(\Omega))$ with $ u_t\in L^2(0,\infty;L^2(\Omega))$, provided that  $u_0 \in \mathcal{W}$ (see \cite[Theorem 2.3]{Zuo}).
\end{remark}

\section{Main result}

 Let $u=u(t)$ be solution of \eqref{problem1} with initial data
$u_0\in X_{0,A}$, then by $T=T(u_0)$ we denote the maximal existence time  of $u=u(t)$ given as
\begin{itemize}
    \item[(i)] $T=\infty$ if $u(t)\in X_{0,A}$ for $t \in [0,\infty)$;
        \item[(ii)] $T=t_{\max}\, (>0)$ if $u(t)\in X_{0,A}$ for $t \in [0,t_{\max})$, $u(t_{\max})\not\in X_{0,A}$.
\end{itemize}

\begin{theorem}\label{theo1} Assume that hypotheses (F), (M1), (M2)
and (P) hold and $p\in(2,2_s^{\ast})$ $(n>2s)$. Let $u=u(t)$ be a
solution of \eqref{problem1} with $u_0\in X_{0,A}$ such
that $J(u_0)\leq d$ and $I(u_0)>0$. If the maximal existence time
is $T=\infty$, then there exists an increasing sequence
$\{t_k\}_{k=1}^{\infty}$ with $t_k\rightarrow\infty$ as
$k\rightarrow\infty$, such that $u(t_k)$ converges to a stationary
solution $v\in X_{0,A}(\Omega)$ of problem \eqref{problem1}, that is
$u(t_k)\rightarrow v$  as $k\rightarrow\infty$.
\end{theorem}

\begin{proof} We show   global in time solutions to   \eqref{problem1} can not blow-up in infinite time, by arguing in three steps.

    \smallskip

\textbf{Step 1: Existence of an increasing sequence $\{t_k\}$.}

 Let
$u=u(t)$ be a solution of problem \eqref{problem1} with $u_0\in
X_{0,A}$ and maximal existence time $T=\infty$. Now, \cite[Theorem
2.3]{Zuo} gives us
$u\in L^{\infty}(0,\infty;L^2(\Omega))\cap L^{\infty}(0,\infty;X_{0,A})
\cap C(0,\infty;L^2(\Omega))$, $u_t\in L^2(0,\infty;L^2(\Omega)).
$
Multiplying equation \eqref{problem1} by $\phi\in X_{0,A}$ and integrating
over $\Omega$, we get
\begin{align}\label{eq1}
 (u'(t),\phi) =-M(\|u\|_{X_{0,A}}^2)\langle
u(t),\phi\rangle_{X_{0,A}}
+\mathcal{R}\int_{\Omega}f(|u(t)|)u(t)\overline{\phi}\ dx.
\end{align}

By \cite[Lemma 3.3]{Zuo} $J(u(t))$ is non-increasing with respect to $t$, hence
\begin{align}\label{eq2}
0\leq J(u(t))\leq J(u_0),
\end{align}
where we use a
contradiction argument to conclude the first inequality in \eqref{eq2}. So, we assume
that there exists a time $t_0$ such that $J(u(t_0))<0$, then by
 \eqref{J} we deduce that
$0>J(u(t_0))=\frac{1}{2}\mathscr{M}(\|u(t_0)\|_{X_{0,A}}^2)
-\int_{\Omega}F(|u(t_0)|)\ dx.$ It follows that
\begin{align}\label{eq3}
\frac{1}{2}\mathscr{M}(\|u(t_0)\|_{X_{0,A}}^2)<\int_{\Omega}F(|u(t_0)|)\
dx.
\end{align}

Combining the Nehari functional $I(\cdot)$, together with assumptions \eqref{f}, (M2), (P) and inequality
\eqref{eq3}, we conclude that
\begin{align*}
I(u(t_0))=&M(\|u(t_0)\|_{X_{0,A}}^2)\|u(t_0)\|_{X_{0,A}}^2-
\mathcal{R}\int_{\Omega}f(|u(t_0)|)|u(t_0)|^2\ dx\\
\leq&\mu\mathscr{M}(\|u(t_0)\|_{X_{0,A}}^2)-\gamma\int_{\Omega}
F(|u(t_0)|)\ dx
<\mu
\mathscr{M}(\|u(t_0)\|_{X_{0,A}}^2)-\frac{\gamma}{2}
\mathscr{M}(\|u(t_0)\|_{X_{0,A}}^2)\\
\leq&(\mu-\frac{p}{2})\mathscr{M}(\|u(t_0)\|_{X_{0,A}}^2)<0 \quad \mbox{(recall that $\gamma \geq p$)}.
\end{align*}
So, at time $t=t_0$ we have $J(u(t_0))<0$ and $I(u(t_0))<0$,
hence $u(t_0)$ is in the so-called unstable set $\mathcal{U}=\{u\in
X_{0,A} :I(u)<0, \, J(u)<d \mbox{ ($d>0$})\}$. By \cite[Theorem
2.4]{Zuo}, $u(t)$ blows-up in a finite time,
which contradicts
$T=\infty$.
Since $J(u(t))$ is non-increasing with respect to $t$, by
\eqref{eq2} we can find $c$ with
$0\leq c\leq J(u_0)$ and such that $J(u(t))\rightarrow c$ as
$t\rightarrow\infty$. Passing to the limit as $t\rightarrow\infty$ in
$
\int_0^t\|u'(s)\|_{L^2(\Omega)}^2\ ds+J(u)=J(u_0),
$
then we get
$c=J(u_0)-\int_0^{\infty}\|u'(s)\|_{L^2(\Omega)}^2\ ds.$
It follows that we can find an increasing sequence
$\{t_k\}_{k=1}^{\infty}$ with $t_k\rightarrow\infty$ as
$k\rightarrow\infty$ satisfying
\begin{align}\label{eq6}
\lim_{k\rightarrow\infty}\|u'(t_k)\|_{L^2(\Omega)}=0.
\end{align}

\textbf{Step 2: Convergence of $\{u(t_k)\}$ to a function $v\in
X_{0,A}$. }

Equation \eqref{eq1} leads to the following
$$
\langle J'(u(t)),\phi\rangle
=M(\|u(t)\|_{X_{0,A}}^2)\langle
u(t),\phi\rangle_{X_{0,A}}-
\mathcal{R}\int_{\Omega}f(|u(t)|)u(t)\overline{\phi}\ dx
=(-u'(t),\phi) \mbox{ for all $\phi\in X_{0,A}$.}
$$
Combining the Schwartz inequality, the
definition of the first eigenvalue $\lambda_1$ of  $(-\Delta)_A^s$ (see \cite[Proposition
3.3]{Vecchi}) and the limit in  \eqref{eq6}, we conclude
 \begin{align}\label{zero}
&\|J'(u(t_k))\|_{X_{0,A}'}=\sup_{\stackrel{\phi\in X_{0,A}}
{\|\phi\|_{X_{0,A}}=1}}\langle J'(u(t_k)),\phi \rangle =
\sup_{\stackrel{\phi\in X_{0,A}}{\|\phi\|_{X_{0,A}}=1}}(-u'(t_k),\phi)\nonumber\\
\leq& \sup_{\stackrel{\phi\in
X_{0,A}}{\|\phi\|_{X_{0,A}}=1}}\|u'(t_k)\|_{L^2(\Omega)}
\|\phi\|_{L^2(\Omega)}\nonumber
\leq\|u'(t_k)\|_{L^2(\Omega)}\sup_{\phi\in
X_{0,A}\setminus\{0\}}\left(\frac{\|\phi\|_{L^2(\Omega)}}{\|\phi\|_{X_{0,A}}}\right)
\\
\leq &\|u'(t_k)\|_{L^2(\Omega)}\left(\displaystyle\inf_{\phi\in
X_{0,A}\setminus\{0\}}\left(\frac{\|\phi\|_{X_{0,A}}}{\|\phi\|_{L^2(\Omega)}}\right)\right)^{-1}
\leq
\frac{1}{\sqrt{\lambda_1}}\|u'(t_k)\|_{L^2(\Omega)}\rightarrow0
\end{align}
as $k\rightarrow\infty$. As usual, by $X_{0,A}'$ we denote the dual space of $X_{0,A}$, hence we can find $c_1>0$,
independent of the index $k$, such that
\begin{align}\label{eq7}
\|J'(u(t_k))\|_{X_{0,A}'}\leq c_1,\ k=1,2,\cdots.
\end{align}

For the Nehari energy functional $I(\cdot)$, the bound from above in \eqref{eq7} and the Young inequality lead to
\begin{equation}\label{eq8}
|I(u(t_k))|\leq|\langle J'(u(t_k)),u(t_k)\rangle| \leq
\|J'(u(t_k))\|_{X_{0,A}'}\|u(t_k)\|_{X_{0,A}}\nonumber \leq
c_1\|u(t_k)\|_{X_{0,A}}
\leq c_2+\frac{(p-2\mu)m_0}{4p\mu}\|u(t_k)\|_{X_{0,A}}^{2\theta}.
\end{equation}

Using  $J(\cdot)$ and $I(\cdot)$, together with  \eqref{eq8} and (M1), we
obtain
{\small \begin{align*}
&J(u(t_k))=\frac{1}{2}\mathscr{M}\left(\|u(t_k)\|_{X_{0,A}}^2\right)
-\int_{\Omega}F(|u(t_k)|)\ dx\\
\geq&\frac{1}{2\mu}M(\|u(t_k)\|_{X_{0,A}}^2)\|u(t_k)\|_{X_{0,A}}^2
-\frac{1}{\gamma}\int_{\Omega}f(|u(t_k)|)|u(t_k)|^2\ dx\\
=&\frac{1}{2\mu}M(\|u(t_k)\|_{X_{0,A}}^2)\|u(t_k)\|_{X_{0,A}}^2
+\frac{1}{p}I(u(t_k))-\frac{1}{p}M(\|u(t_k)\|_{X_{0,A}}^2)
\|u(t_k)\|_{X_{0,A}}^2\\
\geq&\frac{(p-2\mu)}{2p\mu}M(\|u(t_k)\|_{X_{0,A}}^2)\|u(t_k)\|_{X_{0,A}}^2
-c_2-\frac{(p-2\mu)}{4p\mu}m_0\|u(t_k)\|_{X_{0,A}}^{2\theta}\\
\geq&\frac{(p-2\mu)}{2p\mu}m_0\|u(t_k)\|_{X_{0,A}}^{2\theta}
-\frac{(p-2\mu)}{4p\mu}m_0\|u(t_k)\|_{X_{0,A}}^{2\theta}-c_2.
\end{align*}}

This inequality, taking into account the bound from above in \eqref{eq2}, gives us
$J(u_0)+c_2\geq\frac{(p-2\mu)}{4p\mu}m_0\|u(t_k)\|_{X_{0,A}}^{2\theta},$
which implies that
\begin{align}\label{eq9}
\|u(t_k)\|_{X_{0,A}}\leq\Big[\frac{(J(u_0)+c_2)4p\mu}{(p-2\mu)m_0}\Big]^{\frac{1}{2\theta}}, \quad
k=1,2,\cdots.
\end{align}

From \eqref{eq9} and the fact that the embedding
$X_{0,A}\hookrightarrow L^q(\Omega,\mathbb{C})$ is compact for all
$q\in[1,2_s^{\ast})$ (see \cite[Lemma 2.2]{Fiscella}), then there exist an increasing subsequence, still denoted by
$\{t_k\}_{k=1}^{\infty}$, and a function $v\in X_{0,A}$ such that
$u_k:=u(t_k)$ satisfies the following convergences
\begin{align}
&u_k\xrightarrow{w} v\mbox{  in }X_{0,A} \mbox{ as
}k\rightarrow\infty,\label{eq10}\\
&u_k\rightarrow v \mbox{ in }L^q(\Omega,\mathbb{C})\mbox{ for all }
q\in[1,2_s^{\ast}) \mbox{ as }k\rightarrow\infty.\label{eq11}
\end{align}

By \eqref{eq11}, there exist a subsequence, still denoted by
$\{u_k\}_{k=1}^{\infty}$, and a function $w\in
L^q(\Omega,\mathbb{C})$ for all $q\in[1,2_s^{\ast})$, such that
\begin{align}
&u_k(x)\rightarrow v(x)\mbox{ a.e. in }\Omega \mbox{ as
}k\rightarrow\infty,\label{eq12}\\
&\mbox{for all }k, |u_k(x)|\leq w(x)\mbox{ a.e. in
}\Omega.\label{eq13}
\end{align}

\textbf{Step 3: $v$ is a solution to the stationary problem \eqref{problem2}.}

Let us prove that for all $\phi\in X_{0,A}$ we have
\begin{align}\label{eq14}
\mathcal{R}\int_{\Omega}f(|u_k|)u_k\overline{\phi}\
dx\to\mathcal{R}\int_{\Omega}f(|v|)v\overline{\phi}\ dx \quad \mbox{as $k\rightarrow\infty$.}
\end{align}
 By
\eqref{eq12} and  $f\in C^1([0,\infty))$, we get
$
f(|u_k|)u_k\overline{\phi}\rightarrow f(|v|)v\overline{\phi}, $
for a.e. $x\in\Omega$ as $k$ to $\infty$.
Using \eqref{f} and \eqref{eq13} we
get
$
|f(|u_k|)u_k\overline{\phi}|\leq pC|w|^{p-2}|w||\phi|
=pC|w|^{p-1}|\phi|,$ for a.e. $x\in\Omega$.
The H\"older inequality implies
{\small $$ \int_{\Omega}pC|w|^{p-1}|\phi|\ dx\leq  c\|\phi\|_{L^p}\left(\int_{\Omega}
|w|^{(p-1)\frac{p}{p-1}}\ dx\right)^{\frac{p-1}{p}}=
c\|\phi\|_{L^p}\|w\|_{L^p}^{p-1}\leq c,
$$ }thanks to $w\in
L^q(\Omega,\mathbb{C})$ for all $q\in[1,2_s^{\ast})$ and
 $\phi\in X_{0,A}\hookrightarrow
L^p(\Omega,\mathbb{C})$ for $p\in(2,2_s^{\ast})$. We  conclude
$
|w|^{p-1}|\phi|\in L^1(\Omega).
$
The Lebesgue dominated convergence theorem gives us
\eqref{eq14}. We show that
\begin{align}\label{eq17.1}
\langle J'(u_k),\phi\rangle=M(\|u_k\|_{X_{0,A}}^2)
\langle
u_k,\phi\rangle-\mathcal{R}\int_{\Omega}f(|u_k|)u_k\overline{\phi}\
dx,
\end{align}
 converges to
$0=M(\|v\|_{X_{0,A}}^2)
\langle
v,\phi\rangle-\mathcal{R}\int_{\Omega}f(|v|)v\overline{\phi}\ dx $
 as $k\rightarrow\infty$ and $u_k\rightarrow v$ in $X_{0,A}$. We show that
$M(\|u_k\|_{X_{0,A}}^2)\rightarrow M(\|v\|_{X_{0,A}}^2)$ and $u_k\rightarrow v$ strongly in $X_{0,A}$ (see  \cite{Zhang}).
Since $M$ is continuous, \eqref{eq9}  implies
$M(\|u_k\|_{X_{0,A}}^2)\leq c$ for all  $k\in\mathbb{N}$, some $c>0$, hence $\{M(\|u_k\|_{X_{0,A}}^2)\}_{k\in
\mathbb{N}}$ is bounded  in $\mathbb{R}$. Now, there is a subsequence, still say
$\{M(\|u_k\|_{X_{0,A}}^2)\}_{k\in
\mathbb{N}}$, converging to  $\overline{M}$, so
$\lim_{k\rightarrow\infty}M(\|u_k\|_{X_{0,A}}^2)\langle
v,\phi\rangle_{X_{0,A}} =\overline{M}\langle
v,\phi\rangle_{X_{0,A}}$
and
$$\lim_{k\rightarrow\infty}\int\int_{\mathbb{R}^{2n}}\left[
M(\|u_k\|_{X_{0,A}}^2)-\overline{M}\right]^2
\frac{|v(x)-e^{i(x-y)\cdot
A\left(\frac{x+y}{2}\right)}v(y)|^2}{|x-y|^{n+2s}}\ dx\ dy=0,$$ that
is, $
M(\|u_k\|_{X_{0,A}}^2)v\rightarrow \overline{M}v\mbox{
 in }X_{0,A}.
$ This together with \eqref{eq10} imply
$\lim_{k\rightarrow\infty}M(\|u_k\|_{X_{0,A}}^2)
\langle u_k,v\rangle_{X_{0,A}}=\overline{M}\langle
v,v\rangle_{X_{0,A}}.$
By \eqref{zero}, \eqref{eq10} and
\eqref{eq14}, we pass to the limit as $k\rightarrow\infty$ in
\eqref{eq17.1} to get
$0=\overline{M}\langle
v,\phi\rangle_{X_{0,A}}-\mathcal{R}\int_{\Omega}f(|v|)v\overline{\phi}\
dx$ for all $\phi\in X_{0,A}$.
For $\phi=v$, we get
$\overline{M}\langle v,v\rangle_{X_{0,A}}=\int_{\Omega}f(|v|)|v|^2\ dx.$ Similar to  \eqref{eq14}, we deduce
$\lim_{k\rightarrow\infty}\int_{\Omega}f(|u_k|)|u_k|^2\ dx=\int_{\Omega}
f(|v|)|v|^2\ dx.$
So, \eqref{eq10}, \eqref{zero} lead to
$
|\langle
J'(u_k),u_k\rangle|\leq\|J'(u_k)\|_{X_{0,A}'}\|u_k\|_{X_{0,A}}\\
\leq \|J'(u_k)\|_{X_{0,A}'}\left[\frac{(J(u_0)+c_2)4p\mu}{(p-2\mu)m_0}
\right]^{\frac{1}{2\theta}} \to 0$ as $k\rightarrow\infty$.
We first deduce that
\begin{align*}
&\lim_{k\rightarrow\infty}M(\|u_k\|_{X_{0,A}}^2)\langle
u_k,u_k \rangle_{X_{0,A}}=\lim_{k\rightarrow\infty}(\langle
J'(u_k),u_k\rangle+\int_{\Omega}f(|u_k|)|u_k|^2\ dx)\\
&=\int_{\Omega}f(|v|)|v|^2\ dx=\lim_{k\rightarrow\infty}M(
\|u_k\|_{X_{0,A}}^2)\langle u_k,v\rangle_{X_{0,A}},
\end{align*}
then
$
\lim_{k\rightarrow\infty}M(\|u_k\|_{X_{0,A}}^2)(
\langle u_k,u_k\rangle_{X_{0,A}}-\langle
u_k,v\rangle_{X_{0,A}})=0,
$
and so
\begin{align*}
0=&\lim_{k\rightarrow\infty}M(\|u_k\|_{X_{0,A}}^2)
\langle u_k,u_k-v\rangle_{X_{0,A}}\\
=&\lim_{k\rightarrow\infty}M(\|u_k\|_{X_{0,A}}^2)\left[
\|u_k-v\|_{X_{0,A}}^2+\langle v,u_k-v\rangle_{X_{0,A}}\right]\\
=&\lim_{k\rightarrow\infty}M(\|u_k\|_{X_{0,A}}^2)\|u_k
-v\|_{X_{0,A}}.
\end{align*}
Since $M(\sigma)\geq m_0\sigma^{\theta-1}$ for all $\sigma\geq0$,
then
$\lim_{k\rightarrow\infty}\|u_k-v\|_{X_{0,A}}=0.$ So, $u_k(x)\rightarrow v(x)$ in
$X_{0,A}$, which implies that
$\|u_k\|_{X_{0,A}}^2\rightarrow\|v\|_{X_{0,A}}^2$ as $k\rightarrow\infty$. Using the
continuity of $M$, we get
$\lim_{k\rightarrow\infty}M(\|u_k\|_{X_{0,A}}^2)= M(\|v\|_{X_{0,A}}^2),$
which allows us to conclude that
$\overline{M}=M(\|v\|_{X_{0,A}}^2)$.

\end{proof}


\end{document}